\documentclass[letterpaper, 10 pt, conference]{ieeeconf}  

\IEEEoverridecommandlockouts                              
\usepackage{graphicx}
\usepackage{setspace}

\newtheorem{Lemma}{Lemma}[section]

\newtheorem{Assumption}{Assumption}[section]
\newtheorem{Problem}{Problem}[section]
\newtheorem{Remark}{Remark}[section]
\usepackage{subfig}
\usepackage{cite}
\usepackage{color}
\usepackage{amsmath,bm}
\usepackage{amsmath}
\usepackage{amssymb}
\newtheorem{theorem}{Theorem}

\usepackage{amsmath,amssymb,mathrsfs}

\usepackage[english]{babel}
\usepackage[utf8x]{inputenc}
\usepackage[T1]{fontenc}
\newcommand{\m}[1]{\mathbf{#1}}
\newcommand{\mc}[1]{\mathcal{#1}}
\newcommand{\mb}[1]{\mathbb{#1}}
\newcommand*{\tran}{^{\mkern-1.5mu\mathsf{T}}}

\usepackage{graphicx}
\usepackage[colorinlistoftodos]{todonotes}
\usepackage[colorlinks=true, allcolors=blue]{hyperref}

\title{\LARGE \bf
Pointing consensus for rooted out-branching graphs
}

\author{Minh Hoang Trinh$^{*}$, Daniel Zelazo$^{\dagger}$, Quoc Van Tran$^{*}$,  and Hyo-Sung Ahn$^{*}$
\thanks{$^{*}$School of Mechanical Engineering, Gwangju Institute of Science and Technology (GIST), Gwangju, Republic of Korea. Emails: \{trinhhoangminh,tranvanquoc,hyosung\}@gist.ac.kr}
\thanks{${}^{\dagger}$Faculty of Aerospace Engineering, Technion - Israel Institute of Technology, Haifa 32000, Israel. E-mail: dzelazo@technion.ac.il}
}

\begin{document}

\maketitle

\begin{abstract}
Given a network of multiple agents, the \emph{pointing consensus problem} asks all agents to point toward a common target. This paper proposes a simple method to solve the pointing consensus problem in the plane. In our formulation, each agent does not know its own position, but has information about its own heading vector expressed in a common coordinate frame and some desired relative angles to the neighbors. By exchanging the heading vectors via a communication network described by a rooted out-branching graph and controlling the angle between the heading vectors, we show that all agents' heading vectors asymptotically point towards the same target for almost all initial conditions. Simulations are provided to validate the effectiveness of the proposed method.
\end{abstract}

\section{INTRODUCTION}
Recently, there has been a large amount of research on the consensus algorithm and its applications. Given $n$ agents having different initial state values, by exchanging and updating the states based on the weighted sum of differences, all agents' states  eventually  reach a same value\cite{MesbahiEgerstedt}. The states of the agents could be auxiliarry variables used for decision and control tasks, or physical variables such as positions or velocities in the rendezvous and formation control problems\cite{olfati2007consensuspieee}. 

Unlike the usual consensus algorithm in the literature, the pointing consensus problem requires all agents in a network to direct their heading vectors to a common point in the space. This problem is motivated from applications in camera networks, satellite formations, and antenna arrays. For example, pointing consensus is important in coordinating multiple collectors and combiner spacecrafts in synthetic aperture radars (SAR)\footnote{https://dst.jpl.nasa.gov/control/} for space missions such as earth observation and studying evolution of black holes or other planets\cite{moreira2013tutorial,krieger2009earth}.

In the literature, there are not many works studying control strategies to solve the pointing consensus problem. The authors in \cite{zhang2014distributed} proposed a distributed concurrent targeting control strategy for linear arrays of point sources. The proposed control strategy in \cite{zhang2014distributed} relies on two main assumptions: (i) the agents' positions are collinearly located on a line in a two-dimensional space, and (ii) at the beginning, two agents at two ends of the line already pointed toward the target. However, these assumptions are quite strict when applied to satellite formations since satellites usually do not line-up perfectly. 

In this paper, we firstly formulate a different framework to study the pointing consensus problem. In our setup, each agent can be positioned freely in the plane and can control its heading direction around its position. The position of each agent and the target are not given, but all agents' local reference frames are aligned. The agent has information about its heading vector and some desired relative angles with its neighbors. Further, each agent can receive  the heading vectors from its neighboring agents. The inter-agent communication is described by a fixed and rooted out-branching graph. The information of the common target point is given to each agent in the form of a  desired heading direction and some subtended angles between the heading vectors. More specifically, the agent at the root of the communication graph knows the direction to the target. The other agents know some relative angles that their heading needs to achieve with regard to its neighbors' heading vectors such that if all these angles are satisfied, their heading vectors will target a common point. Secondly, we propose a decentralized control strategy for all agents to target a desired common point. Since the graph is rooted out-branching, the dynamics of the $n$-agent system has a cascade structure. Using notions of almost-global input-to-state stability, we show that all agents' headings will point towards the same target point for almost all initial conditions under the proposed control strategy. Thus, comparing with \cite{zhang2014distributed}, the control strategy in this paper relaxes the assumption on collinearity of all agents and further does not require two agents to specify the target from beginning. Finally, we provide a numerical simulation of a six-agent system to illustrate the control strategy. Since the heading direction of each camera can be modeled as a unit vector, there is an interesting link between the pointing consensus problem with the bearing-only navigation\cite{Loizou2007,Trinh2016CEP} and bearing-based formation control/network localization problems \cite{Bishop2014,zhao2015tac,zhao2016aut} in the literature. It is also worth noting that performing a consensus on the agents' heading vectors leads to the orientation alignment/attitude synchronization problem\cite{ren2007distributed}, or more generally, consensus on nonlinear manifolds \cite{sarlette2009consensus}. In order to solve the pointing consensus problem, beside the local heading directions, we need some relative information between the agents' positions and a common pointing target. 

The rest of this paper is organized as follows. In section \ref{section:2}, the pointing consensus problem is formulated. The control strategy is proposed, analyzed, and discussed in section \ref{section:3}. Section \ref{section:4} provides a simulation result of a six-agent system. Finally, we summarize the paper and discuss further research directions in section \ref{section:5}. 
\subsection{Notations and Preliminaries}
In this paper, lower-case characters are used to denote scalars, while bold-font lower-case (capital) letters and calligraphic letters denote vectors (matrices) and sets, respectively. The rotation matrix 
\begin{equation*}
\m{R}(\alpha) = \begin{bmatrix}
\cos \alpha & -\sin\alpha\\
\sin\alpha & \cos\alpha
\end{bmatrix},
\end{equation*}
rotates points in the plane counterclockwise through an angle $\alpha$ about the origin of the coordinate system. Given a vector $\m{v} \in \mb{R}^2$, the result of rotating $\m{v}$ by an angle $\alpha$ is $\m{R}(\alpha) \m{v}$. For $2 \times 2$ rotation matrix, we have the following properties: $\m{R}(\alpha)^{-1} = \m{R}(\alpha)\tran$, $\|\m{R}(\alpha)\| = 1$, and $\m{R}(\alpha_1)\m{R}(\alpha_2) = \m{R}(\alpha_2)\m{R}(\alpha_1),~\forall \alpha_1, \alpha_2 \in \mb{R}$.

Let $\mc{G} = (\mc{V},\mc{E})$ denote a directed graph with the vertex set $\mc{V} = \{1, \ldots, n\}$ and the edge set $\mc{E} \subset \mc{V} \times \mc{V}$. A directed edge $(j,i)$ in the graph describes that $i$ receives information from $j$ and not vice verse. The in-neighbor set of a vertex $i$ is defined as $\mc{N}_i = \{ j \in \mc{V}: (j, i) \in \mc{E}\}$.  A directed path is a sequence of vertices $i_1i_2\ldots i_p$ such that $(i_l, i_{l+1}) \in \mc{E}$. A directed cycle is a directed path having the same start and end vertices, i.e, $i_1 \equiv i_p$. A directed acyclic graph is a graph without any directed cycle. If there exists a vertex (called the root) such that for any vertex $i$ in the graph, there exists a directed path from the root to this vertex $i$, the graph is called rooted out-branching.

\section{PROBLEM FORMULATION}
\label{section:2}
\begin{figure}[tb]
\centering
\includegraphics[height = 2.5cm]{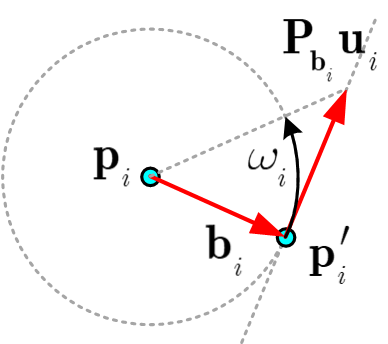}
\caption{The rotational velocity $\omega_i$ of the camera $i$'s heading direction can be equivalently represented as the velocity $\m{u}_i = \m{P}_{\m{b}_i}\m{u}_i$ of the head point $\m{p}_i'$.}
\label{fig:model}
\end{figure}

Consider a network of $n$ agents in a two-dimensional ambient space, each agent $i$ $(i=1, \ldots, n)$ is located at a fixed position $\m{p}_i \in \mb{R}^2$ (i.e., $\dot{\m{p}}_i = \m{0}$). The agent $i$ has a heading direction described by a unit vector $\m{b}_i \in \mb{R}^2$. 
Suppose that each agent can control its own heading direction by rotating its heading direction around the point $\m{p}_i$ with an angular velocity $\omega_i$, as illustrated in Fig.~\ref{fig:model}. Define $\m{p}_i' = \m{p}_i + \m{b}_i$. It can be seen from Fig.~\ref{fig:model} that the rotational motion of the heading direction is equivalent to the motion of the point $\m{p}_i'$, which is perpendicular to the heading direction $\m{b}_i$. Let $\m{P}_{\m{b}_i}:= \m{I}_2 - \m{b}_i \m{b}_i\tran$ denote the orthogonal projection matrix corresponding to $\m{b}_i$. Note that $\m{P}_{\m{b}_i} = \m{P}_{\m{b}_i}\tran = \m{P}_{\m{b}_i}^2$. Further, $\m{P}_{\m{b}_i}$ is positive semidefinite, and $\mc{N}(\m{P}_{\m{b}_i}) = \text{span}(\m{b}_i)$ \cite{zhao2015tac}. The control effort to change $\m{b}_i$ can be given as
\begin{equation}
\dot{\m{p}}_i' = \m{P}_{\m{b}_i} \m{u}_i,
\end{equation}
where $\m{u}_i \in \mb{R}^2$ will be designed later. Then, we have
\begin{equation} \label{eq:bearing_dot}
\dot{\m{b}}_i = \m{P}_{\m{b}_i} \m{u}_i.
\end{equation}
Note that a relationship between $\omega_i$ and $\m{u}_i$ is given as follows $\| \m{P}_{\m{b}_i}\m{u}_i \| = |\omega_i| \|\m{b}_i\| = |\omega_i|$. 
\begin{figure}
\centering
\includegraphics[height = 3.86cm]{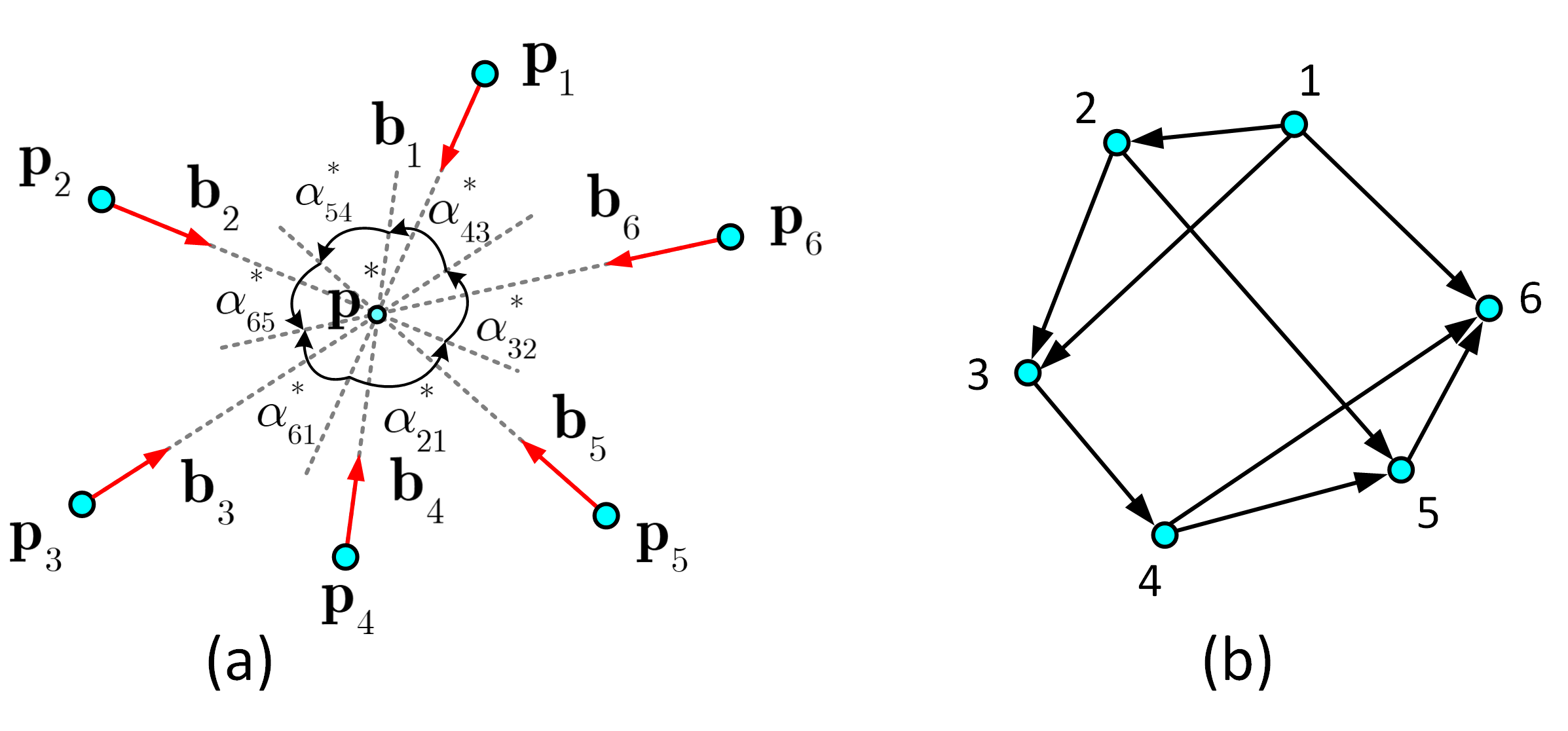}
\caption{A six-agent system. (a) The agents want to consent their heading direction into a common point $\m{p}^*$. (b) The information graph $\mc{G}$ is a connected directed acyclic graph (rooted at vertex 1).}
\label{fig:consensus}
\end{figure}
\subsection{The pointing consensus problem}
In order to focus the headings on a common point, each agent needs to exchange its heading information with its neighbors through a directed graph $\mc{G} = (\mc{V},\mc{E})$. 
In this paper, we assume that $\mc{G}$ is a rooted out-branching. Without loss of generality, we can label the vertices of $\mc{G}$ such that vertex 1 is the root of the graph. Suppose that we want all agents to point to a given target in the plane. To guide all headings, we assume that agent 1 knows the direction to the target. Each agent $i~(i \geq 2)$ is given a set of desired angles $\alpha_{ij}^*\in (-\pi,\pi]~(\forall j \in \mc{N}_i)$. Here, $\alpha_{ij}^*$ is the angle between $\m{b}_j$ and $\m{b}_i$ when $i$ and $j$ are pointing at the target. For example, a six-agent system is depicted in Fig.~\ref{fig:consensus}.

Obviously, if agent $i$ receives $\m{b}_{j}$ from agent $j$, it can calculate the difference between the two vectors $\m{b}_i$ and $\m{R}(\alpha_{ij}^*)\m{b}_{j}$, i.e., $i$ can calculate $\m{b}_i- \m{R}(\alpha_{ij}^*)\m{b}_{j}$. 
Then, agent $i$ can control its heading correspondingly to reduce the angle error $\|\m{b}_i- \m{R}(\alpha_{ij}^*)\m{b}_{j}\|$. We list all assumptions as follows:

\begin{Assumption} \label{assumption:1}
The information graph $\mc{G} = (\mc{V}, \mc{E})$ is rooted out-branching. Vertex 1 is the root of the graph. 
\end{Assumption}
\begin{Assumption} \label{assumption:2}
All agents' local reference frames are aligned. Agent 1 knows its desired heading vector $\m{b}_{1}^*$, the other agents know their desired angles $\alpha_{ij}^*$ and receive $\m{b}_j$ from all $j \in \mc{N}_i$. The set of desired subtended angles $\{\alpha_{ij}^*\}_{(i,j) \in \mc{E}}$ is feasible. That is, there exists $\m{p}^* \in \mb{R}^2$ s.t.
\begin{align}
 \m{b}_1^* &= ({\m{p}^* -\m{p}_{1}})/{\|\m{p}^* -\m{p}_{1}\|}, \label{cond:1}\\
\alpha_{ij}^* &= \angle(\m{p}^* - \m{p}_j,~\m{p}^* -\m{p}_{i}),~\forall i, j = 2, \ldots, n. \label{cond:2}
\end{align}
\end{Assumption}
\begin{figure}
\centering
\includegraphics[height=3.43cm]{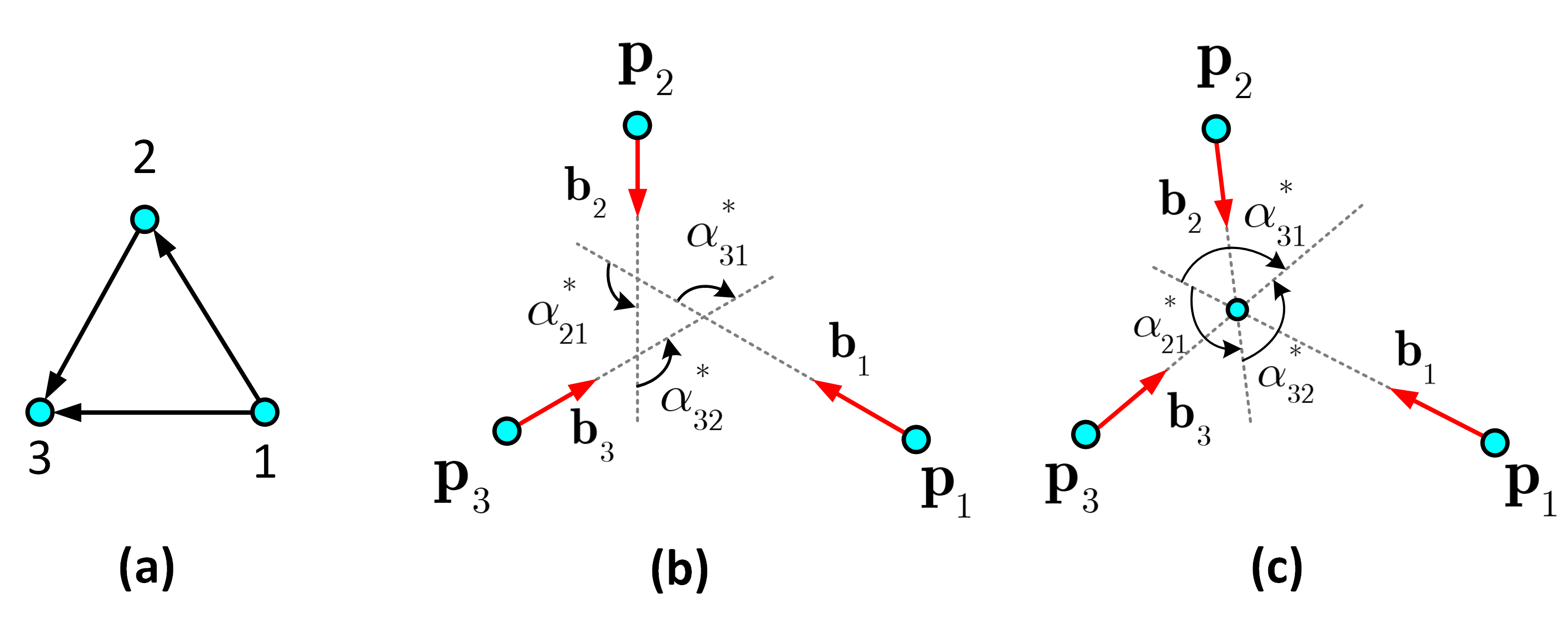}
\caption{A three-agent system with the communication graph (a) and the desired angles $\alpha_{21}^* = \alpha_{32}^* = -\alpha_{31}^* = \frac{2\pi}{3}$. In (b), although all desired angles are satisfied, the agents do not target a common point. In (c), by changing the direction of $\m{b}_1$, a pointing consensus is achieved.} 
\label{fig:uniqueness}
\end{figure}

In Assumption \ref{assumption:2}, it is remarked that the condition \eqref{cond:2} is equivalent to having $\m{p}^* \in \mb{R}^2$ such that:
\begin{equation} \label{cond:2b}
\frac{\m{p}^* - \m{p}_i} {\| \m{p}^* - \m{p}_i \| } = \m{R}(\alpha_{ij}^*) \frac{\m{p}^* - \m{p}_j}{\| \m{p}^* - \m{p}_j \|}, \forall (i, j) \in \mc{E}.
\end{equation}
The assumption that agent 1 knows $\m{b}_1^*$ is important to guide all agents' heading vectors to point to a common target. Miscontrolling the heading of the agent 1 will lead the other agents to not point to a common target even though all the subtended desired angles are satisfied. Consider Fig.~\ref{fig:uniqueness}, a three-agent system located at the vertices of an acute triangle $\m{p}_1, \m{p}_2, \m{p}_3$. The desired subtended angles are selected as $\alpha_{21}^* = \alpha_{32}^* = -\alpha_{31}^* = \frac{2\pi}{3}$.\footnote{From elementary geometry, we know that at the Torricelli point of the triangle, we have $\angle(\m{b}_1,\m{b}_2)= \angle(\m{b}_2,\m{b}_3) = \angle(\m{b}_3,\m{b}_1) = \frac{2 \pi}{3}~(*)$. The Torricelli point is also the only point in the plane that satisfies $(*)$.} When the heading $\m{b}_1$ does not point into the Torricelli point (Fig.~\ref{fig:uniqueness}(b)), there is a configuration where all desired angles are satisfied, but the heading vectors do not target a common point. This ambiguity does not happen in Fig.~\ref{fig:uniqueness}(c), when $\m{b}_1$ points toward the Torricelli point of the triangle.

Finally, we make an assumption for the later analysis.
\begin{Assumption} \label{assumption:3}
The initial heading vector of agent 1 satisfies
$\m{b}_1(0)  \neq - \m{b}_{1}^*.$ 
Also, the desired angles are such that $\alpha_{21}^* \notin \{0, -\pi~(\text{mod } 2\pi)\}$, 
\end{Assumption}

We can now state the main problem in this paper:
\begin{Problem} \label{Problem:1}
Given the $n$-agent system satisfying Assumption \ref{assumption:1}--\ref{assumption:3}, design a decentralized control law such that all agents' heading target a common point as $t \to \infty$.
\end{Problem}

\section{MAIN RESULTS}
\label{section:3}
\subsection{The proposed control law}
We propose the following control law for each agent to solve the pointing consensus problem with rooted out-branching graphs:
\begin{subequations}
\begin{align} 
\dot{\m{p}}_1' &= \m{P}_{\m{b}_i} \m{b}_1^*, \label{eq:control_law0}\\
\dot{\m{p}}_i' &= \m{P}_{\m{b}_i} \sum_{j \in \mc{N}_i} \m{R}(\alpha_{ij}^*)\m{b}_{j},~\forall i = 2, \ldots, n. \label{eq:control_law}
\end{align}
\end{subequations}
\begin{figure}[t!]
\centering
\subfloat[]{%
\includegraphics[height = 3cm]{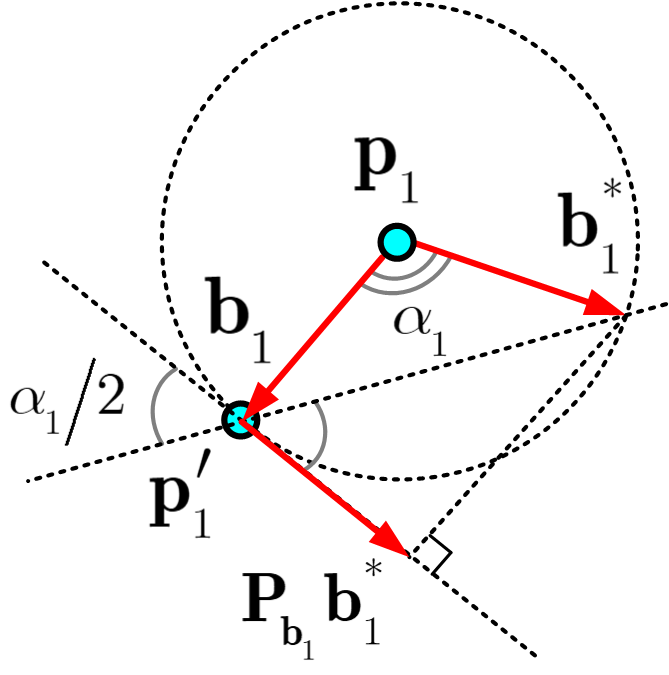}
\label{fig:agent1}}
\qquad
\subfloat[]{%
\includegraphics[height = 3cm]{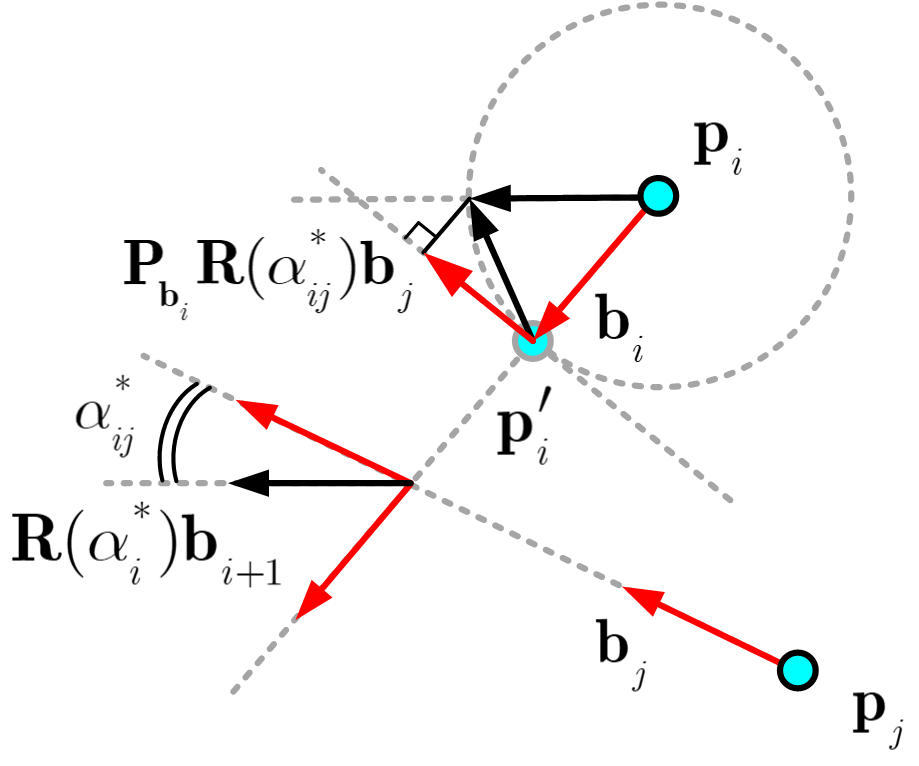}
\label{fig:principle}}
\caption{(a): The control law \eqref{eq:control_law0} steers $\m{b}_1$ toward $\m{b}_{1}^*$. (b):~Suppose that $\m{b}_{j}$ is invariant, the control law \eqref{eq:control_law} steers $\m{b}_i$ toward $\m{R}(\alpha_{ij}^*)\m{b}_{j}$.}
\end{figure}
It is easy to see that $\dot{\m{p}}_i'$ vanishes when $\m{b}_i$ aligns with $\m{R}(\alpha_{ij}^*)\m{b}_{j}$, or the desired subtended angle between two agents $i$ and $j$ is satisfied for all $i \in \{2, \ldots, n\}$. 
Intuitively, when the agent $j$'s heading $\m{b}_{i+1}$ is fixed, the control law \eqref{eq:control_law} steers $\m{b}_{i}$ toward $\m{R}(\alpha_{ij}^*)\m{b}_{j}$ (see Figs.~\ref{fig:agent1}--\ref{fig:principle}). We refer readers to \cite{zhao2015tac, Minh2016_ifaclss, Trinh2018} for some related explanations of the control law \eqref{eq:control_law}. 

From equation~\eqref{eq:bearing_dot}, we can also write the heading direction dynamics as follows:
\begin{subequations} \label{eq:system}
\begin{align} 
\dot{\m{b}}_1 & = \m{P}_{\m{b}_i} \m{b}_1^*, \label{eq:system1}\\
\dot{\m{b}}_i &= \m{P}_{\m{b}_i}\sum_{j \in \mc{N}_i}\m{R}(\alpha_{ij}^*) \m{b}_j,~i = 2, \ldots, n. \label{eq:system2}
\end{align}
\end{subequations}
We have the following remark on the control law \eqref{eq:control_law}: 
\begin{Remark}
Suppose that instead of receiving the heading direction from agent $j$, agent $i$ can measure the heading direction $\m{b}_j^i$ of agent $j$ in its local reference frame $^i\sum$, which is identified by the rotation matrix $\m{R}_i$ with regard to the global reference frame. The control law can be written in the local reference frame of agent $i$ as follows:
\begin{align}
\dot{\m{b}}_i^i = \m{P}_{\m{b}^i_i} \sum_{j \in \mc{N}_i} \m{R}(\alpha_{ij}^*) \m{b}_j^i, \label{eq:local_law}
\end{align}
where $\m{b}_{i}^i = \m{R}_{i}\tran \m{b}_i,~\m{b}_{j}^i = \m{R}_{i}\tran \m{b}_j$. Observe that $\m{P}_{\m{b}^i_i} = \m{I}_2 - \m{b}_{i}^i (\m{b}_{i}^{i})\tran = \m{I}_2 - \m{R}_{i}\tran \m{b}_i \m{b}_i\tran \m{R}_{i} = \m{R}_{i}\tran\m{P}_{\m{b}_i}\m{R}_{i}$, and $\m{R}(\alpha_{ij}^*)\m{R}_i\tran = \m{R}_i\tran \m{R}(\alpha_{ij}^*)$. Substituting into \eqref{eq:local_law}, we get
\begin{align}
\dot{\m{b}}_i^i = \m{R}_{i}\tran\m{P}_{\m{b}_i}\sum_{j \in \mc{N}_i} \m{R}(\alpha_{ij}^*) \m{b}_j, \label{eq:local_law1}
\end{align}
Thus, the control law \eqref{eq:local_law1} written in the global reference frame is 
$\dot{\m{b}}_i = \m{R}_{i} \dot{\m{b}}_i^i = \m{P}_{\m{b}_i}\sum_{j \in \mc{N}_i}\m{R}(\alpha_{ij}^*) \m{b}_j, $
which is exactly the same as \eqref{eq:system2}. Hence, if the agents can sense the relative heading of its neighbors, \eqref{eq:system2} does not require the agents' reference frames $^2\sum, \ldots, ^n\sum$ to be aligned.
\end{Remark}
\subsection{Analysis}
In this subsection, we study the system \eqref{eq:system}. Firstly, we examine the equilibrium set of \eqref{eq:system1}. We will prove that the proposed control law solves the pointing consensus problem for almost all initial conditions. It will be shown that the $n$-agent system has the form of a cascade system. Then, we adopt the notions of almost global  input-to-state stability to establish the convergence result. 

\begin{Lemma} \label{lem:agent1} The equilibrium set of \eqref{eq:system1} is: 
\begin{align}
\mc{E}_1 = \{\m{b}_1 \in \mb{R}^{2}|~ \m{b}_1 = \pm \m{b}_1^*\}.
\end{align}
The equilibrium $\m{b}_1 = \m{b}_1^*$ is almost globally exponentially stable while the equilibrium $\m{b}_1 = -\m{b}_1^*$ is unstable.
\end{Lemma}

\begin{proof}
Any equilibrium point of \eqref{eq:system1} must satisfy 
$\m{P}_{\m{b}_1} \m{b}_{1}^* = \m{0}.$
Since $\mc{N}(\m{P}_{\m{b}_i}) = \text{span}(\m{b}_1)$ and $\|\m{b}_1^*\| = 1$, it follows that $\m{b}_{1} = \pm \m{b}_1^*$. We examine the stability of the equilibrium $\m{b}_1 = -\m{b}_1^*$ by linearization. Since
$\frac{\partial \dot{\m{b}}_1}{\partial \m{b}_1} = \frac{\partial }{\partial \m{b}_1} \left( \m{b}_1^* - \m{b}_1\m{b}_1\tran\m{b}_1^* \right)
= -\m{b}_1\tran\m{b}_1^* \m{I}_2 - \m{b}_1 (\m{b}_1^{*})\tran,$
it follows that $\m{A} = \frac{\partial \dot{\m{b}}_1}{\partial \m{b}_1}|_{\m{b}_1 = -\m{b}_1^*} = \m{I}_2 + \m{b}_1^* (\m{b}_1^{*})\tran$ is positive definite. Thus, the equilibrium $\m{b}_1 = -\m{b}_1^*$ is (exponentially) unstable. 

Next, consider the potential function $V = \frac{1}{2} \|\m{b}_{1} - \m{b}_{1}^*\|^2$, which is continuously differentiable. Moreover, $V \geq 0$ and $V = 0$ if and only if $\m{b}_{1} = \m{b}_{1}^*$. The derivative of $V$ along a trajectory of \eqref{eq:system1} is
\begin{align}
\dot{V} &= (\m{b}_{1} - \m{b}_{1}^*)\tran \dot{\m{b}}_{1} = (\m{b}_{1} - \m{b}_{1}^*)\tran \m{P}_{\m{b}_{1}} \m{b}_{1}^* \nonumber\\
& = -(\m{b}_{1}^*)\tran \m{P}_{\m{b}_{1}} \m{b}_{1}^* = -  \|\m{P}_{\m{b}_{1}} \m{b}_{1}^*\|^2 \leq 0. \label{eq:lem2-1}
\end{align}
Obviously, $\dot{V} = 0$ if and only if $\mathbf{b}_{1} = \pm \mathbf{b}_{1}^*$. Since the equilibrium $-\mathbf{b}_{1}^*$ is unstable, for all $\mathbf{b}_1(0) \neq \mathbf{b}_{1}^*$,  ${\m{b}_{1}} \to \m{b}_{1}^*$ due to LaSalle's invariance principle. Further, let $\alpha_{1}$ be the angle between $\m{b}_1$ and $\m{b}_1^*$. For $\m{b}_1(0) \neq - \m{b}_1^*$, we have $\alpha_{1} \in [0, \pi)$, and $ \|\m{P}_{\m{b}_{1}}(\m{b}_{1} - \m{b}_{1}^*)\| =  \|\m{b}_{1} - \m{b}_{1}^*\| |\cos \left(\frac{\alpha_1}{2}\right)|$. 
As a result, since $\alpha_1(t) \to 0$ as $\m{b}_1 \to \m{b}_1^*$ and the $\cos(\cdot)$ function is decreasing in $[0, \frac{\pi}{2})$, we have
\begin{equation}\label{eq.expstab1}
\dot{V} \leq -\cos^2 \frac{\alpha_1(t)}{2} \|\m{b}_{1} - \m{b}_{1}^*\| \leq -\cos^2 \frac{\alpha_1(0)}{2} V.
\end{equation}
Thus, $\m{b}_{1} = \m{b}_{1}^*$ is almost globally exponentially stable.
\end{proof}
\begin{Lemma} \label{lem:agent2} Under Assumptions \ref{assumption:1}--\ref{assumption:3}, the equilibrium set of agent 2 is: 
\begin{align}
\mc{E}_2 = \{\m{b}_2 \in \mb{R}^{2} |~ \m{b}_2 = \pm \m{R}(\alpha_{21}^*) \m{b}_1^*\}.
\end{align}
The equilibrium $\m{b}_2 = \m{R}(\alpha_{21}^*) \m{b}_1^*$ is almost globally asymptotically stable while the equilibrium $\m{b}_2 = -\m{R}(\alpha_{21}^*) \m{b}_1^*$ is unstable.
\end{Lemma}
\begin{proof}
Consider the agent 2's heading dynamics:
\begin{equation} \label{eq:agent2_dynam}
\dot{\m{b}}_2 = \m{P}_{\m{b}_2} \m{R}(\alpha_{21}^*) \m{b}_1,
\end{equation}
Observe that the equilibria of \eqref{eq:agent2_dynam} depends on the agent 1's equilibria. We therefore rewrite the dynamics of agents 1 and 2 in form of a cascade system as follows: 
\begin{subequations} \label{eq:cascade}
\begin{align}
\dot{\m{b}}_1 &= \m{P}_{\m{b}_1} \m{b}_1^* = \m{f}_1(\m{b}_1), \label{eq:cascade-1}\\
\dot{\m{b}}_2 &= \m{P}_{\m{b}_2} \m{R}(\alpha_{21}^*)\m{b}_1 = \m{f}_2(\m{b}_1,\m{b}_2).\label{eq:cascade-2}
\end{align}
\end{subequations}
In \eqref{eq:cascade}, $\m{b}_1$ is an input to the system \eqref{eq:cascade-2}. Further, when $\m{b}_1 = \m{b}_1^*$, we may express \eqref{eq:cascade-2} as follows:
\begin{align} \label{eq:unforce-2}
\dot{\m{b}}_2 = \m{f}_2(\m{b}_1^*,\m{b}_2) = \m{P}_{\m{b}_2} \m{R}(\alpha_{21}^*)\m{b}_1^* = \m{P}_{\m{b}_2} \m{b}_2^*,
\end{align}
which has the same form as \eqref{eq:cascade-1}. From Lemma \ref{lem:agent1}, the equilibria of \eqref{eq:agent2_dynam} satisfies $\m{b}_2 = \pm \m{R}(\alpha_{21}^*) \m{b}_1^*$. Moreover, the equilibrium $\m{b}_2 = \m{b}_2^*$ of \eqref{eq:unforce-2} is almost globally exponentially stable while $\m{b}_2 = -\m{b}_2^*$ is isolated and unstable. Letting $\m{b}_2^* = \m{R}(\alpha_{21}^*) \m{b}_1^*$, we consider the function $V = \frac{1}{2}\|\m{b}_2 - \m{b}_2^*\|^2$ which is continuously differentiable. Moreover, $V \geq 0$ and $V = 0$ if and only if $\m{b}_2 = \m{b}_2^*$. The derivative of $V$ along a trajectory of \eqref{eq:agent2_dynam} is given as follows:
\begin{align}
\dot{V} & = (\m{b}_2 - \m{b}_2^*)\tran \m{P}_{\m{b}_2} \m{R}(\alpha_{21}^*) \m{b}_1 \nonumber\\
& = -\m{b}_2^{*T} \m{P}_{\m{b}_2} (\m{b}_2^* + \m{R}(\alpha_{21}^*)(\m{b}_1 - \m{b}_1^*)) \nonumber\\
& \leq - \|\m{P}_{\m{b}_2} \m{b}_2^{*}\|^2 + \|\m{P}_{\m{b}_2} \m{b}_2^{*}\| \|\m{b}_1 - \m{b}_1^*\|. \label{eq:agent2-ultimate}
\end{align}
It follows from \eqref{eq.expstab1} that there exists $\delta_1,\eta_1 > 0$ such that $\|\m{b}_1 - \m{b}_1^*\| \leq \delta_1 e^{-\eta_1 t}$. Thus, 
\begin{align}
\dot{V} &\leq - \|\m{P}_{\m{b}_2} \m{b}_2^{*}\|^2 + \delta_1 e^{-\eta_1 t} \|\m{P}_{\m{b}_2} \m{b}_2^{*}\| \leq \frac{\delta_1^2}{4} e^{-2 \eta_1 t}, 
\end{align}
where the inequality holds if and only if $\|\m{P}_{\m{b}_2} \m{b}_2^{*}\| = \frac{\delta_1}{2} e^{-\eta_1 t}$. Thus, there holds
\begin{equation}
V(\infty) - V(0) \leq \int_{0}^\infty \frac{\delta_1^2}{4} e^{-2 \eta t} dt = \frac{\delta_1^2}{2 \eta_1},
\end{equation}
which shows that $V$ is bounded. Therefore, the system \eqref{eq:cascade-2} satisfies the ultimate boundedness property \cite{angeli2011}[Proposition 3], and thus \eqref{eq:cascade-2} is also almost globally Input-to-State-Stable (ISS) with regard to the input $\m{b}_1$. 

It follows from $\|\m{b}_1 - \m{b}_1^*\| \to 0$ (Lemma \ref{lem:agent1}) and \cite{angel2004}[Theorem 2] that the equilibrium $\m{b}_2 = \m{b}_2^*$ of \eqref{eq:cascade-2} is almost globally asymptotically stable.
\end{proof}

We have the following remark on the uniqueness of the common point from Lemma \ref{lem:agent2}.
\begin{Remark} \label{remark:p_star}
The headings $\m{b}_1^*$ and $\m{b}_2^*$ of agents 1 and 2 at their equilibrium states uniquely determine a point $\m{p}^* \in \mb{R}^2$. Indeed, the equations
\begin{align*}
\frac{\m{p}^* - \m{p}_1}{\|\m{p}^* - \m{p}_1\|} =  \m{b}_{1}^*,~\text{and }
\frac{\m{p}^* - \m{p}_2}{\|\m{p}^* - \m{p}_2\|} = \m{R}(\alpha_{21}^*)\m{b}_{1}^* := \m{b}_{2}^*,
\end{align*}
imply that
$\m{P}_{\m{b}_1^*}\frac{\m{p}^* - \m{p}_1}{\|\m{p}^* - \m{p}_1\|} = \m{P}_{\m{b}_1^*} \m{b}_{1}^* = \m{0}, $ and 
$\m{P}_{\m{b}_2^*}\frac{\m{p}^* - \m{p}_2}{\|\m{p}^* - \m{p}_2\|} = \m{P}_{\m{b}_2^*}\m{b}_{2}^* = \m{0}.$
Thus, $(\m{P}_{\m{b}_1^*} + \m{P}_{\m{b}_2^*}) \m{p}^* = \m{P}_{\m{b}_1^*} \m{p}_1 + \m{P}_{\m{b}_2^*} \m{p}_2$ and for $\alpha_{21}^* \notin \{0, \pi\}$, (or i.e., $\m{b}_i^* \neq \m{b}_j^*$) we can uniquely determine
\begin{equation} \label{eq:p_star}
\m{p}^* = (\m{P}_{\m{b}_1^*} + \m{P}_{\m{b}_2^*})^{-1} (\m{P}_{\m{b}_1^*} \m{p}_1 + \m{P}_{\m{b}_2^*} \m{p}_2),
\end{equation}
which is the common point for all agents' heading direction.
\end{Remark}


Now, consider an arbitrary agent $i~(i \geq 3)$ with the heading dynamics as given in \eqref{eq:system2}. The system \eqref{eq:system} can be written in the following form:
\begin{subequations} \label{eq:agent-1-i}
\begin{align}  
\dot{\m{b}}_{[1:i-1]} &= \m{f}_{[1:i-1]}(\m{b}_{[1:i-1]}) \\
\dot{\m{b}}_i &= \m{f}_i(\m{b}_{[1:i-1]}, \m{b}_i), \label{eq:agent-i-f}
\end{align}
\end{subequations}
where $\m{b}_{[1:i-1]} = [\m{b}_1\tran, \ldots, \m{b}_{i-1}\tran]\tran$ can be considered as an input to the system \eqref{eq:agent-i-f}.

\begin{Lemma} \label{lem:agent_i} Under Assumptions \ref{assumption:1}--\ref{assumption:3}, given that all agents $1, \ldots, i-1$'s heading vectors are pointing toward $\m{p}^*$ determined from \eqref{eq:p_star} , i.e., $\m{b}_j = \m{b}_j^*$, $\forall j = 1, \ldots, i-1$, the heading $\m{b}_i$ of agent $i$~($i \geq 2$) converges asymptotically to $\m{b}_i = \m{b}_i^* := \frac{\m{p}^* - \m{p}_i}{\|\m{p}^* - \m{p}_i\|}$ satisfying $\m{b}_i^* = \m{R}(\alpha_{ij}^*) \m{b}_j^*, \forall j \in \mc{N}_i$ if initially $\m{b}_i(0) \neq -\m{b}^*_i$.
\end{Lemma}
\begin{proof}
 Assumption \ref{assumption:2} guarantees the existence of the point $\m{p}^*$ which is uniquely determined from \eqref{eq:p_star}. From condition \eqref{cond:2b}, for
$\m{b}_i^* = \frac{\m{p}^* - \m{p}_i}{\|\m{p}^* - \m{p}_i\|},$
we conclude that $\m{b}^*_i$ satisfies $\m{b}_i^* = \m{R}(\alpha_{ij}^*) \m{b}_j^*,~\forall j \in \mc{N}_i.$
Further, the dynamics of $\m{b}_i$ when $\m{b}_j = \m{b}_j^*, \forall j = 1, \ldots, i-1$ can be  rewritten as follows:
\begin{align} \label{eq:agent-i}
\dot{\m{b}}_i &= \m{P}_{\m{b}_i} \sum_{j \in \mc{N}_i} \m{R}(\alpha_{ij}^*) \m{b}_j^*=|\mc{N}_i| \m{P}_{\m{b}_i} \m{b}_i^*.
\end{align}
The equilibria of \eqref{eq:agent-i-f} is the solution of $\dot{\m{b}}_i = \m{f}_i(\m{b}_1^*, \ldots, \m{b}_{i-1}^*, \m{b}_i) = \m{0}.$
This implies $\m{b}_i = \pm \m{b}_i^*$ are two equilibria of \eqref{eq:agent-i}. The system  \eqref{eq:agent-i} has the same form as \eqref{eq:system1}. Thus, we conclude that $\m{b}_i = \m{b}_i^*$ is an almost globally exponential stable equilibrium of \eqref{eq:agent-i}, while the equilibrium $\m{b}_i = -\m{b}_i^*$ is unstable.
\end{proof}
We can now state the main result of this paper.

\begin{theorem} \label{thm:1}
Under the Assumptions \ref{assumption:1}-\ref{assumption:3}, the headings asymptotically target a common point, i.e., $\m{b}_1 \to \m{b}_1^*$ and $\m{b}_i \to \m{R}(\alpha_{ij}^*) \m{b}_{j}$, $\forall (i,j) \in \mc{E}$ as $t \to \infty$.
\end{theorem}

\begin{proof}
We will show that for all $i = 1, \ldots, n$, the agent $i$'s heading will point to $\m{p}^*$ determined by \eqref{eq:p_star} by mathematical induction. For $i=1, 2$, we have proved in Lemmas \ref{lem:agent1} and \ref{lem:agent2} that $\m{b}_1\to \m{b}_1^*,~\m{b}_2 \to \m{b}_2^*$ asymptotically if $\m{b}_1(0) \neq -\m{b}_1^*$. 

Suppose that the claim holds until $i-1~(i \geq 3)$, i.e., the heading of every agent $j \in \{1, 2, \ldots, i-1\}$ asymptotically points toward $\m{p}^*$. We prove that $\m{b}_i = \m{b}_i^*$ is an asymptotically stable equilibrium of the system \eqref{eq:agent-i-f}. Consider the  potential function $V = \frac{1}{2} \sum_{j \in \mc{N}_i} \|\m{b}_i - \m{R}(\alpha_{ij}^*) \m{b}_{j}\|^2$, which is positive definite, continuously differentiable. Further, $V = 0$ if and only if $\m{b}_i = \m{b}_{i}^*$. For any trajectory of \eqref{eq:agent-1-i} with $\m{b}_1(0) \neq -\m{b}_1^*$, we have
\begin{align}
\dot{V} & = (\m{b}_i - \m{b}_i^*)\tran \m{P}_{\m{b}_i} \sum_{j \in \mc{N}_i} \m{R}(\alpha_{ij}^*) \m{b}_j \nonumber\\
& = -\m{b}_i^{*T} \m{P}_{\m{b}_i} \sum_{j \in \mc{N}_i} \m{R}(\alpha_{ij}^*) \m{b}_j^{*} \nonumber\\
& \qquad -\m{b}_i^{*T} \m{P}_{\m{b}_i} \sum_{j \in \mc{N}_i} \m{R}(\alpha_{21}^*)(\m{b}_j - \m{b}_j^*) \nonumber \\
& = -|\mc{N}_i| \m{b}_i^{*T} \m{P}_{\m{b}_i} \m{b}_i^{*} -\m{b}_i^{*T} \m{P}_{\m{b}_i} \sum_{j \in \mc{N}_i} \m{R}(\alpha_{21}^*)(\m{b}_j - \m{b}_j^*) \nonumber \\
& \leq - |\mc{N}_i| \|\m{P}_{\m{b}_i} \m{b}_i^{*}\|^2 + \|\m{P}_{\m{b}_i} \m{b}_i^{*}\| \sum_{j \in \mc{N}_i} \|\m{b}_j - \m{b}_j^*\|. 
\end{align}
Given any $\epsilon > 0$ such that $\|\m{P}_{\m{b}_i} \m{b}_i^*\| > \epsilon$. Since $\m{b}_j \to \m{b}_j^*$ asymptotically $\forall j = 1, \ldots, i-1$, for a small $\epsilon>0$ there exists a time instance such that $\max_{j=1, \ldots, i-1} \|\m{b}_j - \m{b}_j^*\| < \epsilon$. It follows that
\begin{align}
\dot{V} & < - |\mc{N}_i| \|\m{P}_{\m{b}_i} \m{b}_i^{*}\|^2 + \epsilon |\mc{N}_i| \|\m{P}_{\m{b}_i} \m{b}_i^{*}\| \nonumber\\
& < -|\mc{N}_i|\|\m{P}_{\m{b}_i} \m{b}_i^{*}\|(\|\m{P}_{\m{b}_i} \m{b}_i^{*}\| - \epsilon) < 0. \label{eq:agent-i-ultimate}
\end{align}
Equation \eqref{eq:agent-i-ultimate} shows that the system \eqref{eq:agent-i-f} fulfills the ultimate boundedness property \cite{angeli2011}[Proposition 3]. Together with the result in Lemma \ref{lem:agent_i}, it follows that \eqref{eq:cascade-2} is almost globally ISS with regard to the input $\m{b}_{[1:i-1]}$. It follows from the induction assumptions $\m{b}_{[1:i-1]} \to [\m{b}_1^{*T},\ldots,\m{b}_{i-1}^{*T}]$ and \cite{angel2004}[Theorem 2] that the equilibrium $\m{b}_i = \m{b}_i^*$ of \eqref{eq:cascade-2} is almost globally asymptotically stable.

Finally, since the claim holds for any $i \geq 2$, it is true for $i=n$. Therefore, the heading vectors of $n$ agents asymptotically point toward the common point $\m{p}^*$ determined from \eqref{eq:p_star}.
\end{proof}
\subsection{Discussions}
In this subsection, we have two remarks on our proposed approach to solve the pointing consensus problem. 

\textbf{The assumption on the communication graphs}: In this paper, we only study pointing consensus problem with rooted out-branching graphs. The restriction is made because we do not know how to solve the problem for general graphs. The solution of pointing consensus in this case may give some insight to address the problem in general. The assumption on rooted out-branching graph allows us to define the set of desired angles $\alpha_{ij}^*$. Further, the $n$-agent dynamics has form of a cascade structure under this assumption. Without the rooted out-branching assumption, the $n$-agent dynamics has some equilibrium sets that are hard to analyze.

\textbf{The choose of the set points}: 
Suppose the set points, which contain information on the target (heading directions or angles), are sent to the $n$-agent system from a command center located far from the system. For example, one can consider a satellite formation orbiting the Earth and a central ground control station on Earth sending the data. 
The command center knows the exact positions of the agents as well as the desired target point. From Lemma~ \ref{lem:agent1}, if the command center sends each agent $i$ the direction toward a common target, i.e, $\m{b}_{i}^*$, agent $i$ can directly point to ward the target under the control law \eqref{eq:control_law0} without exchanging any information with its neighbors. 

In this paper, the set points are given as a desired heading $\m{b}_1^*$ and some angles $\alpha_{ij}^*$. Thus, each agent ($2, \ldots, n$) blindly follows the heading directions of its neighbors and controls its heading under \eqref{eq:control_law} accordingly. Although our approach is indirect, it is advantageous in terms of security as explained by the following scenario. Suppose that there is an attacker who knows the positions of all agents $\m{p}_i, \forall i,$ and tries to figure out the target's position by decoding the set points sent from the command center. Based on Remark \ref{remark:p_star}, the attacker can compute the target point if he can decode the information of $\m{b}_1^*$ and $\alpha_{21}^*$ from equation \eqref{eq:p_star}. However, if the attacker cannot decode both $\m{b}_1^*$ and $\alpha_{21}^*$ from the sent information, he cannot find the target point from the remaining set points.\footnote{Note that in the example depicted in Fig.~\ref{fig:uniqueness}, the set of desired angles do not give enough information to determine the target.} Meanwhile, if the set points are given as $\{\m{b}_i^*\}_{\{i=1,...,n\}}$, the attacker can locate the target point if any two desired heading vectors are decoded. For example, given $\m{b}_i^*$, $\m{b}_j^*$, the target point can be computed by $\m{p}^* = (\m{P}_{\m{b}_i^*} + \m{P}_{\m{b}_j^*})^{-1} (\m{P}_{\m{b}_i^*}\m{p}_i + \m{P}_{\m{b}_j^*}\m{p}_j)$. 
Therefore, by giving the set points with one desired heading vector and some desired angles allows a higher privacy level on masking  the target. 
\section{SIMULATIONS}
\label{section:4}
\begin{figure}[t!]
\centering
\subfloat[The initial \& final headings]{%
    \includegraphics[height =3.61cm]{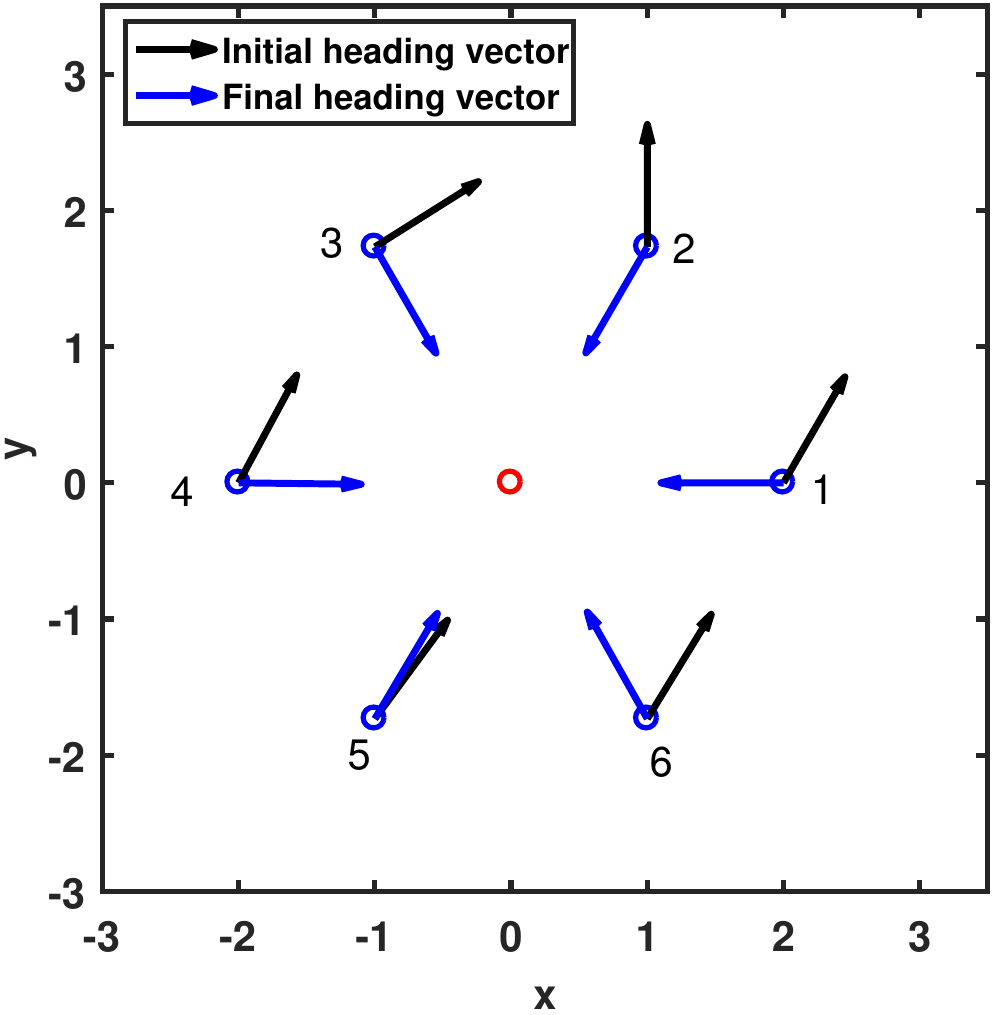}
 	\label{fig:2a}}
 	\quad
\subfloat[Heading vector errors vs Time]{%
    \includegraphics[height =3.61cm]{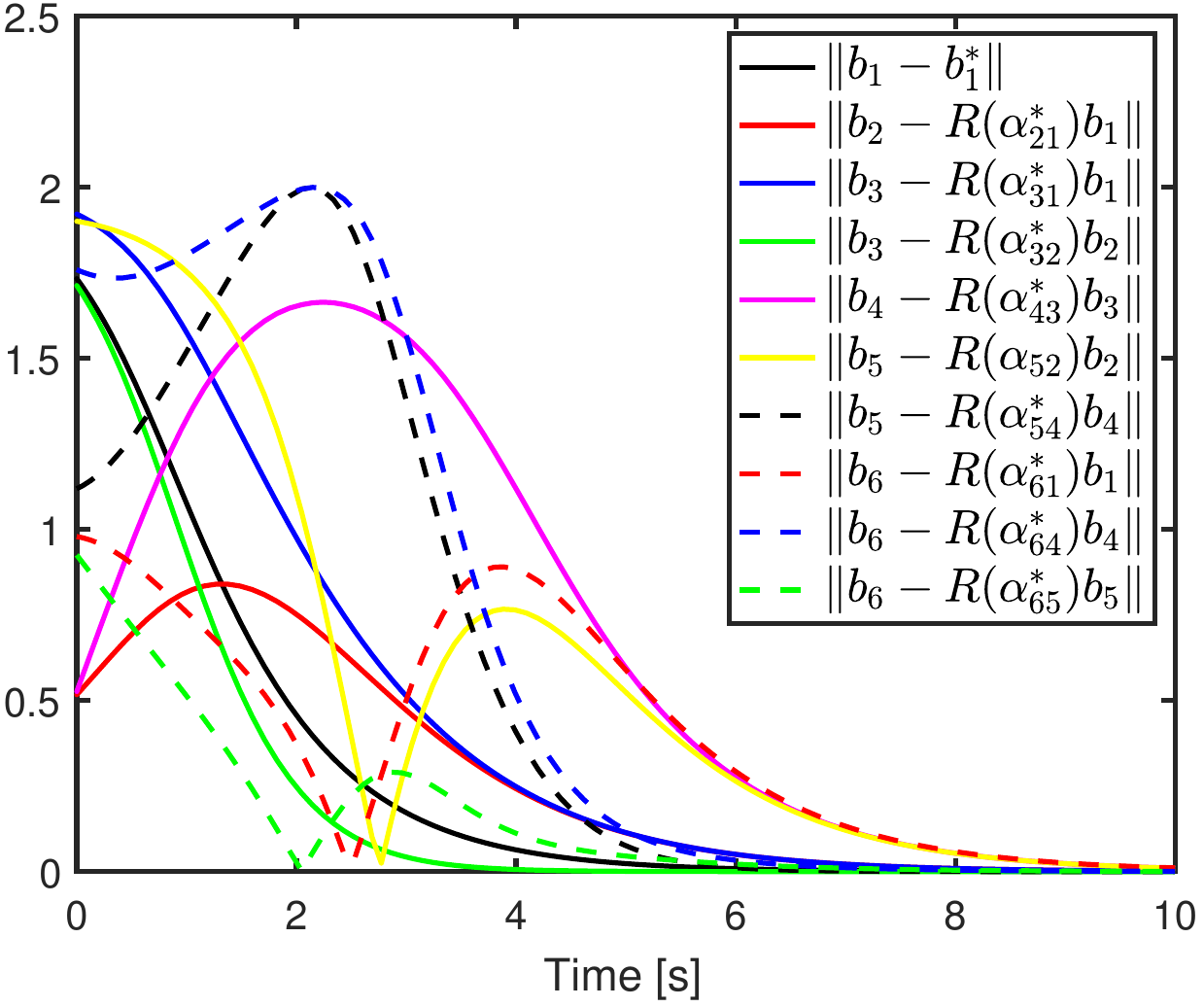}
 	\label{fig:2b}}
\caption{\label{fig:sim} Under the control law \eqref{eq:system}, six agents asymptotically point toward a common point (marked by `\textcolor{red}{o}').}
\end{figure}
Consider a six-agent system with the interaction topology as described by Figure \ref{fig:consensus}.b. Six agents are positioned at the vertices of a regular hexagon: $\m{p}_{1} = [2,0]\tran,~\m{p}_{2} = [1,\sqrt{3}]\tran,~\m{p}_{3} = [-1,\sqrt{3}]\tran,~\m{p}_{4} = [-2,0]\tran,~\m{p}_{5} = [-1,-\sqrt{3}]\tran,~\m{p}_{6} = [1,-\sqrt{3}]\tran$. We aim to control all agents to point toward the origin $[0, 0]\tran$. To this end, the desired heading of agent 1 is set to be $\m{b}_1^* = [-1,~0]\tran$. The desired angles are given as follows: $\alpha_{21}^* = \alpha_{32}^*= \alpha_{43}^* = \alpha_{65}^*= \frac{\pi}{3}$, $\alpha_{31}^* = \alpha_{64}^* = \frac{2\pi}{3}$, and $\alpha_{61}^* = -\frac{\pi}{3}$. 

We simulate the six-agent system under the control law \eqref{eq:system}. The initial heading vectors were randomly chosen. The simulation results are provided in Fig.~\ref{fig:sim}. It can be seen from Fig.~\ref{fig:2a} that all agents' heading vectors gradually point toward the origin. Figure~\ref{fig:2b} shows that the heading vector errors eventually vanish under the control law \eqref{eq:system}. 

Next, we simulate the system with the same initial condition except that the desired heading of agent 1 was changed to $\m{b}_1^* = [-\frac{\sqrt{2}}{2},~\frac{\sqrt{2}}{2}]\tran$, i.e., $\m{b}_1^*$ does not point to the desired target. Simulation results are shown in Fig.~\ref{fig:sim1}. Although all desired angles $\alpha_{ij}^*$ are satisfied (Fig.~\ref{fig:3b}), the agents' headings do not target a common point. Observe from Fig.~\ref{fig:3a} that the intersections of the lines containing the heading directions are vertices of a regular hexagon. 


\section{CONCLUSIONS}
\label{section:5}
In this paper, a framework for studying the pointing consensus problem was formulated. A control strategy was proposed to solve the problem with rooted out-branching graphs. The analysis in this paper was based on mathematical induction and the notion of almost global input-to-state stability theory. 

When the graph is not restricted to directed acyclic graph, there may exist some undesired equilibria that are nontrivial to examine their stability. Studying the pointing consensus with general graphs is a further research direction. It is observed that some relative information on the positions of the agents need to be available to solve this problem completely.


\section*{ACKNOWLEDGMENT}
The work of M. H. Trinh, Q. V. Tran, and H.-S. Ahn was supported by GIST Research Institute (GRI) and by the National Research Foundation (NRF) of Korea under the grant NRF-2017R1A2B3007034.

The work of D. Zelazo was supported in part at the Technion by a fellowship of the Israel Council for Higher Education and the Israel Science Foundation (grant No. 1490/1).

\begin{figure}[t!]
\centering
\subfloat[The initial \& final headings]{%
    \includegraphics[height = 3.61cm]{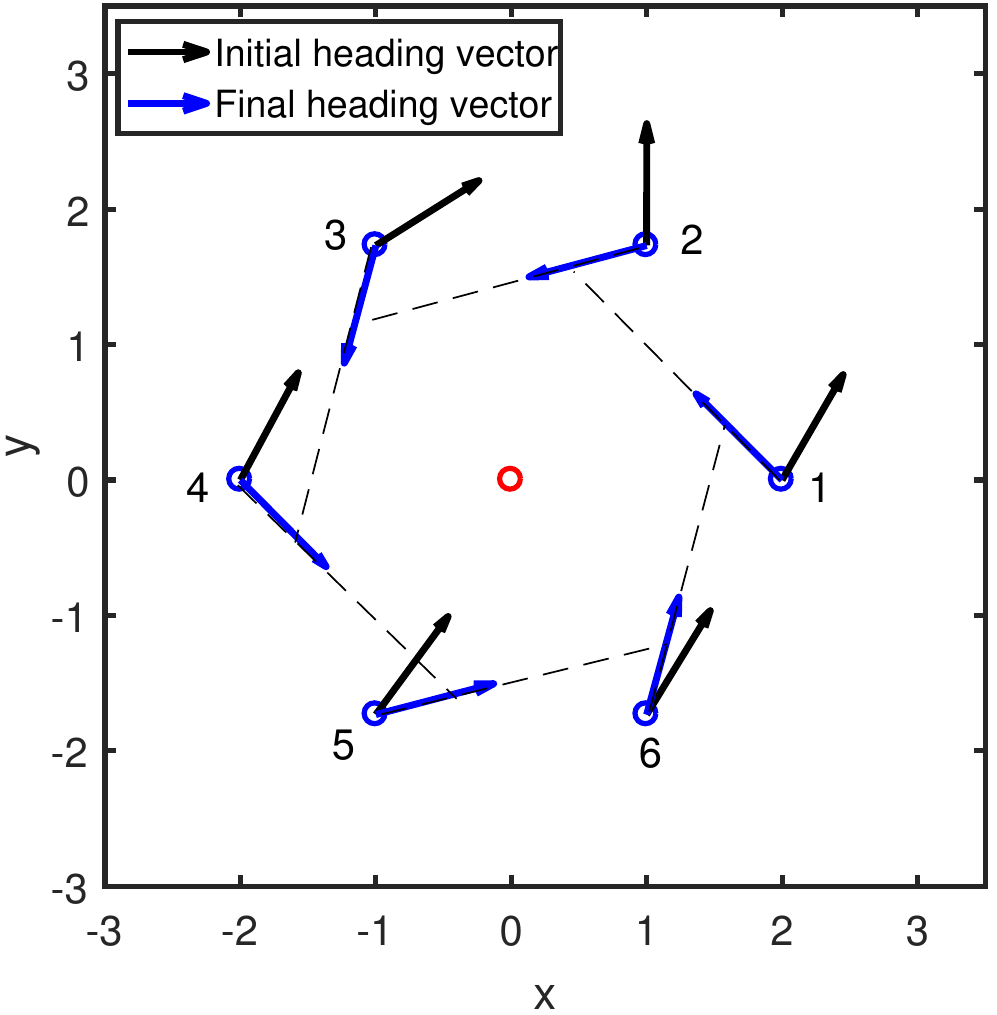}
 	\label{fig:3a}}
 	\quad
\subfloat[Heading vector errors vs Time]{%
    \includegraphics[height = 3.61cm]{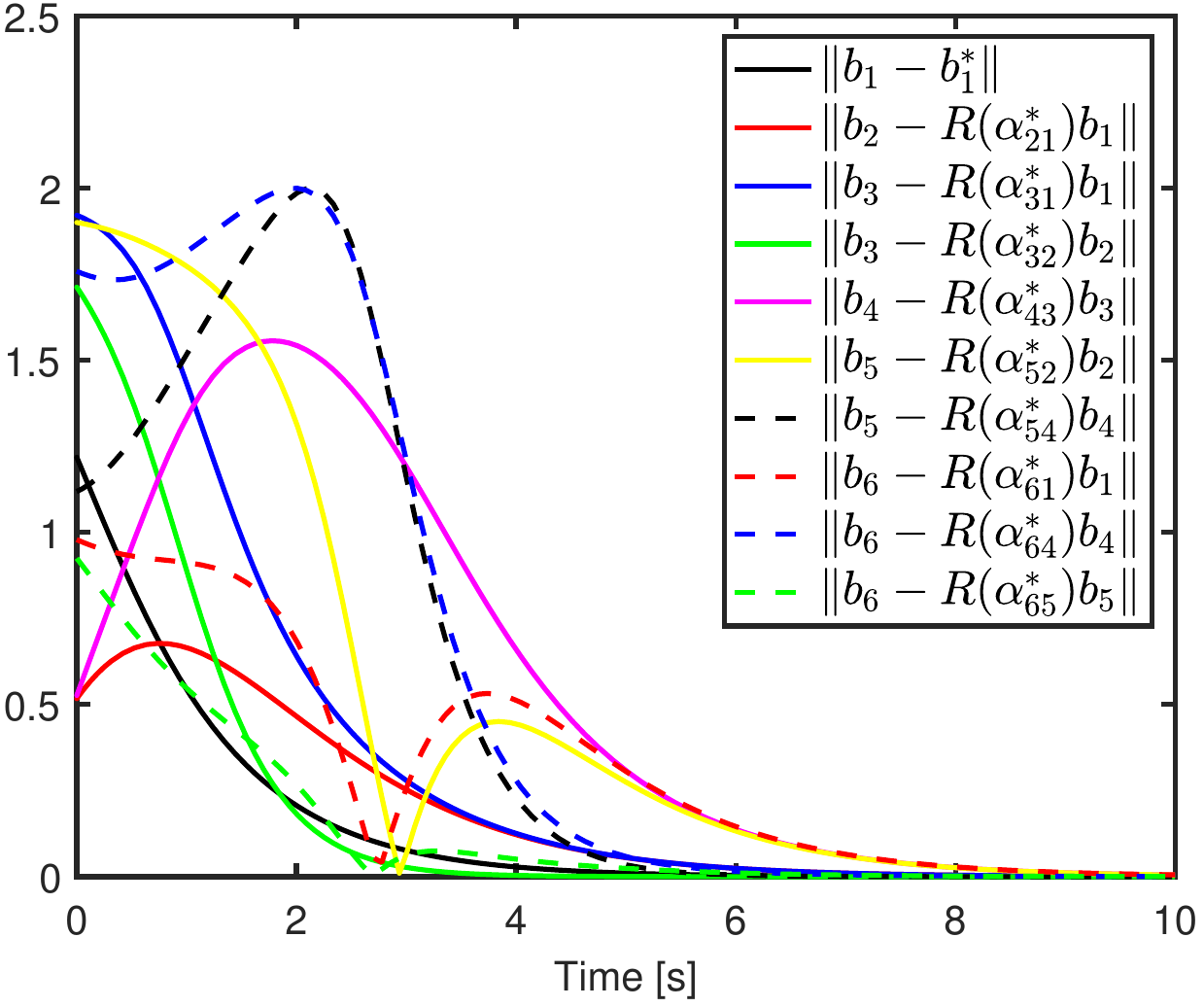}
 	\label{fig:3b}}
\caption{\label{fig:sim1} When agent 1 does not point toward the desired target, six agents do not target a same point.}
\end{figure}

\bibliographystyle{IEEEtranS}
\bibliography{minh2017}     
\addtolength{\textheight}{-12cm}

\end{document}